\documentclass[12pt,a4paper]{article}
\usepackage{natbib}
\usepackage{float}
\usepackage{rotating}
\usepackage{multirow,bbm}
\usepackage{amssymb,amsmath, amsfonts, booktabs}
\usepackage{lscape}
\usepackage{enumerate}
\usepackage{amsthm}
\usepackage[T1]{fontenc}
\usepackage[utf8]{inputenc}
\usepackage{subcaption}
\usepackage{lmodern}
\usepackage[english]{babel}
\usepackage{pdflscape}
\usepackage{csquotes}
\usepackage[colorlinks, citecolor=blue]{hyperref}
\usepackage{mathtools}

\usepackage{xcolor, natbib}
\numberwithin{equation}{section}
\renewcommand{\baselinestretch}{1}
\usepackage{graphicx}
\usepackage{caption}
\newtheorem{theorem}{Theorem}[section]
\newtheorem{exe}{Example}[section]

\newtheorem{c1}{Corollary}[section]

\setcitestyle{authoryear,open={(},close={)}} 

\begin{document}
\renewcommand{\baselinestretch}{1.2}
\renewcommand{\theequation}{\thesection.\arabic{equation}}
	\begin{center}
	{\Large\bf A Study on Cumulative Residual Extropy of Linear Consecutive k-out-of-n:G Systems} \\
	\vspace{0.3in}
	
	{\bf Aman Pandey\footnote{Corresponding author e-mail: amanp23ma@rgipt.ac.in, amanp23m@gmail.com~(Aman Pandey).}}~~~~~~~\textbf{Chanchal Kundu}
	\\
	Department of Mathematical Sciences\\ Rajiv Gandhi Institute of Petroleum Technology\\Jais, U.P., India, 229304\\
\begin{center}
    {May, 2026}
\end{center}
\end{center}
\begin{abstract}
In this article, we investigate the cumulative residual extropy associated with linear consecutive $k$-out-of-$n{:}G$ systems, which play an important role in reliability theory and engineering applications. We first derive explicit expressions for the proposed measure and examine the behavior of cumulative residual extropy  under a variety of stochastic orderings. In addition, several  bounds and meaningful results of the characterization are established. Moreover, we also introduced the dynamic version of the cumulative residual extropy and explored the relationship between the proposed dynamic version and the mean residual life function.
From an inferential perspective, we develop a nonparametric estimation procedure for the cumulative residual extropy and establish the corresponding consistency properties of the estimator. The finite-sample performance of the proposed estimator is further investigated through extensive Monte Carlo simulation studies under different parametric settings and validated through a real dataset. 
\end{abstract}
\section{Introduction}
Entropy and its numerous extensions have long served as fundamental quantitative tools for measuring uncertainty and information associated with random phenomena. Since the seminal contribution of \cite{shannon1948}, entropy-based methodologies have found wide-ranging applications in reliability theory, statistical physics, information science, econometrics, and biological sciences. In recent years, the notion of \emph{extropy}, introduced by \cite{lad2015}, has attracted considerable attention as a complementary counterpart to entropy, offering an alternative framework for assessing uncertainty.

Let $X$ be an absolutely continuous non-negative random variable (RV) with probability density function (PDF) $f(w)$ and survival function (SF) $\bar F(w)$. The differential extropy of $X$ is defined as
\begin{align}\label{eq1.1}
    \mathfrak{J}(X)
    =
    -\frac{1}{2}\int_{0}^{\infty} f^{2}(w)\,dw,
\end{align}
provided that the integral exists. In contrast to Shannon entropy, which quantifies uncertainty through logarithmic information content, extropy captures the concentration characteristics of the underlying distribution by means of the squared PDF. Owing to its mathematical simplicity and interpretational relevance, extropy and its generalized forms have become increasingly important in information theory and reliability analysis.

However, in many practical situations PDF may either fail to exist or may not be readily estimable from observed data. Motivated by this limitation, \cite{jahan2019} introduced the concept of \emph{cumulative residual extropy} (CREx) by replacing PDF in \eqref{eq1.1} with SF $S(w)$. The CREx measure is defined as
\begin{align}\label{eq1.2}
    J(X)
    =
    -\frac{1}{2}\int_{0}^{\infty} \bar F^{2}(w)\,dw,
\end{align}
where $\bar F(w)=1-F(w),$ $F(w)$ is the cumulative distribution function (CDF).
The CREx measure possesses several advantages over differential extropy. Since the SF is generally more regular than PDF, CREx remains well-defined even in situations where the PDF is unavailable or difficult to handle analytically. Furthermore, from a statistical perspective, the CDF and  SF can be conveniently estimated through the empirical distribution function (EDF), and such estimators are known to be consistent. Consequently, CREx admits natural empirical estimators based on the EDF, whereas the differential extropy cannot be directly estimated in a similar manner.

Subsequently, \cite{sathar2021} investigated dynamic versions of CREx and studied several characterization properties together with nonparametric estimation procedures. Further developments concerning estimation methodologies and inferential aspects of CREx may be found in \cite{kaz2021}, \cite{kattu2022}, \cite{pakda2025}, and \cite{chak2025}. Although entropy-based measures have been extensively utilized in a variety of applied domains, the practical potential of extropy and CREx has not yet been explored to the same extent. Motivated by this observation, the present work employs the CREx measure to address several problems arising in reliability engineering.

The \emph{linear consecutive $k$-out-of-$n$} (denoted by $\mathcal C(k|n{:}G)$) systems play an important role in reliability theory due to their ability to model several practical engineering systems. A linear consecutive $\mathcal C(k|n{:}G)$ system consists of $n$ linearly ordered components and the system functions if and only if at least $k$ consecutive components are operating. Such systems arise naturally in many real-life applications. For instance, in a pipeline network, uninterrupted flow requires several adjacent pipeline segments to function simultaneously. In telecommunication systems, successful signal transmission may depend on the operation of consecutive relay stations along a communication line. Similarly, in railway track monitoring systems, consecutive functioning sensors are required for reliable detection and control.

The linear $\mathcal C(k|n{:}G)$ model generalizes the classical series and parallel systems. In particular, when $k=n$, the system operates only if all components function, which corresponds to a series system. On the other hand, when $k=1$, the system functions whenever at least one component operates, yielding a parallel system. Hence, consecutive systems provide a flexible framework for modeling complex reliability structures.

Consider a linear $\mathcal C(k|n{:}G)$ system consisting of $n$ components arranged sequentially, where the system operates whenever at least $k$ consecutive components are functioning. Let $X_1,X_2,\ldots,X_n$ denote the lifetimes of the components, assumed to be independent and identically distributed (iid) with common SF $S(w)=P(X>w)$. The lifetime of the system is denoted by $T_{k|n:G}$. A particularly important case arises when $2k\geq n$. Under this condition, it is impossible for two non-overlapping runs of $k$ functioning components to occur simultaneously, which considerably simplifies the reliability analysis of the system. Using this observation, \cite{ery2009} obtained the SF of $\mathcal C(k|n{:}G)$ system as
\begin{align}\label{eq1.3}
    \bar F_{k|n:G}(w)
    =
    (n-k+1)\bar F^{k}(w)
    -
    (n-k)\bar F^{k+1}(w),
    \quad w>0,
\end{align}
where $S(w)$ represents the common SF of the components.

The information-theoretic properties of reliability systems and order statistics have attracted considerable attention in the literature. \cite{qiu2019} investigated the extropy of mixed systems and derived an explicit expression for the extropy of the system lifetime, while also introducing the Jensen--extropy divergence as an information measure for comparing mixed systems with homogeneous components. Later, \cite{chak2024} studied CREx for coherent and mixed systems with both iid  and dependent and identically distributed  components. They established several bounds, proposed divergence and discrimination measures for comparing system structures, and analyzed coherent systems with heterogeneous components in redundancy allocation problems.

More recently, \cite{kayid2024} investigated the Shannon differential entropy of the $\mathcal C(k|n{:}G)$ systems by deriving explicit expressions, characterization results, and nonparametric estimators for the proposed entropy measure. Subsequently, \cite{kayid2025} studied the cumulative residual entropy of the $\mathcal C(k|n{:}G)$ systems and obtained explicit expressions together with several important properties. Additional developments on the $\mathcal C(k|n{:}G)$ systems can be found in \citet{dem2025}, \cite{tian2026}, \cite{yi2026}, and \cite{ery2026}, and the references therein.

To the best of our knowledge,  CREx of the $\mathcal C(k|n{:}G)$ systems has not yet been investigated in the literature. Motivated by the increasing interest in information-theoretic measures for reliability systems, as well as the recent developments on entropy and extropy measures for coherent and consecutive systems, we study the CREx associated with the lifetime of the $\mathcal C(k|n{:}G)$ systems. Since consecutive systems arise naturally in reliability engineering, communication networks, and industrial applications, understanding their uncertainty behavior through CREx is of considerable theoretical and practical importance.

In this paper, we establish an explicit representation for the CREx of the lifetime of a $\mathcal C(k|n{:}G)$ system. Using the proposed representation, we establish several comparison results under different stochastic orders together with a number of useful bounds and characterization results. The obtained results enrich the existing literature on information measures in reliability theory and provide further insights into the uncertainty characteristics of consecutive systems.

The remainder of the paper is structured as follows. Section~2 introduces the CREx measure for the $\mathcal C(k|n{:}G)$ systems and investigates its properties under various stochastic orders, along with several associated bounds and characterization results. Section~3 develops a nonparametric estimation procedure and establishes its consistency, supplemented by a comprehensive simulation study and real data analysis. In Section~4, an application of the proposed CREx measure in image processing is presented. The dynamic version of CREx introduces for the $\mathcal C(k|n{:}G)$ systems in Section 5. Finally, Section~6 summarizes the key findings of the study and outlines potential avenues for future research.
\section{CREx of consecutive \texorpdfstring{$\mathcal C({k|n{:}G})$}~~\text{system}}
In this section, we introduce CREx associated with the $\mathcal C(k|n{:}G)$ systems and investigate several of its structural properties, including bounds and characterization results. 
Our analysis starts with the derivation of an explicit expression for the CREx of the system lifetime $T_{k|n:G}$, assuming that the component lifetimes are iid RVs having an absolutely continuous CDF $F$. 
The obtained representation serves as a foundation for the subsequent theoretical developments presented in this section.

Define
$
V_{k|n:G}=F(T_{k|n:G}),
$
and for each component lifetime $X_i$, let
$
V_i=F(X_i), \quad i=1,\ldots,n.
$
Since $F$ is absolutely continuous, the RVs $V_1,\ldots,V_n$ are iid standard uniform RVs on $(0,1)$.

Proceeding similarly to \eqref{eq1.3}, for  $2k\geq n$, the SF of $V_{k|n:G}$ is given by
\begin{align}\label{eq:2.1}
\bar G_{k|n:G}(v)
=
(n-k+1)(1-v)^k
-(n-k)(1-v)^{k+1},
\end{align}
for $0<v<1$.

Using the definition of CREx together with the transformation $v=F(w)$, we obtain
\begin{align}
J(T_{k|n:G})
&=
-\frac12\int_0^\infty
\Phi\!\left(
\bar F_{k|n:G}(w)
\right)\,dw \nonumber\\
&=
-\frac12\int_0^\infty
\Phi\!\left(
(n-k+1)\bar F^k(w)
-(n-k)\bar F^{k+1}(w)
\right)\,dw,
\label{eq:2.2}
\end{align}
where $\Phi(t)=t^2, ~0<t<1$.
Now, employing the probability integral transformation $v=F(w)$ together with
$
dw=\dfrac{dv}{f(F^{-1}(v))}
$
and $\bar F(w)=1-v$ in \eqref{eq:2.2}, we get
\begin{align}
J(T_{k|n:G})
&=
-\frac12
\int_0^1
\frac{
\Phi\!\left(
(n-k+1)(1-v)^k
-(n-k)(1-v)^{k+1}
\right)
}{
f(F^{-1}(v))
}\,dv \nonumber\\
&=
-\frac12
\int_0^1
\frac{
\Phi\!\left(
\bar G_{k|n:G}(v)
\right)
}{
f(F^{-1}(v))
}\,dv.
\end{align}
The last equality is obtained by using \eqref{eq:2.1}.

In the sequel, we illustrate CREx with the following example.
\begin{exe}
Let $T_{k|10{:}G}$ represents the lifetime of a $\mathcal C(k|10{:}G)$ system for $k=2,3,4,5,$ and $6$, where each component lifetime characterize by a Lomax distribution with CDF
    \begin{align*}
        F(w)=1-\exp\left(1-\left(\frac{w}{\eta}\right)^{\beta}\right),~w>0,~\beta>0,~\eta>0.
    \end{align*}
    The corresponding PDF evaluated at $F^{-1}(v)$ is $f(F^{-1}(v))=\frac{\beta}{\eta}(1-v)\left(-\log(1-v)\right)^{\frac{\beta-1}{\beta}}$.Obtaining a closed-form analytical expression for CREx is mathematically intractable. Therefore, we investigate its behaviour numerically. In Figure~\ref{fig1}(a), we fix $\beta=2.5$ and vary the scale parameter over $\eta\in(0,4)$. Similarly, in Figure~\ref{fig1}(b), we fix $\eta=2.5$ and consider the shape parameter in the range $\beta\in(1,10)$. These numerical results illustrate the influence of the Weibull distribution parameters on the CREx measure.

 Since CREx is an extropy-based measure, smaller values of CREx correspond to greater uncertainty in the system behaviour. From Figure~\ref{fig1}(a), it is observed that CREx decreases as the scale parameter $\eta$ increases for all considered values of $k$, indicating an increase in uncertainty and hence a decrease in predictability of the system lifetime. Likewise, Figure~\ref{fig1}(b) shows that CREx decreases with increasing values of the shape parameter $\beta$, which also reflects higher uncertainty and lower predictability.
\begin{figure}[htbp]
    \centering
    \begin{minipage}{0.45\textwidth}
        \centering
        \includegraphics[width=7.5cm,height=6cm]{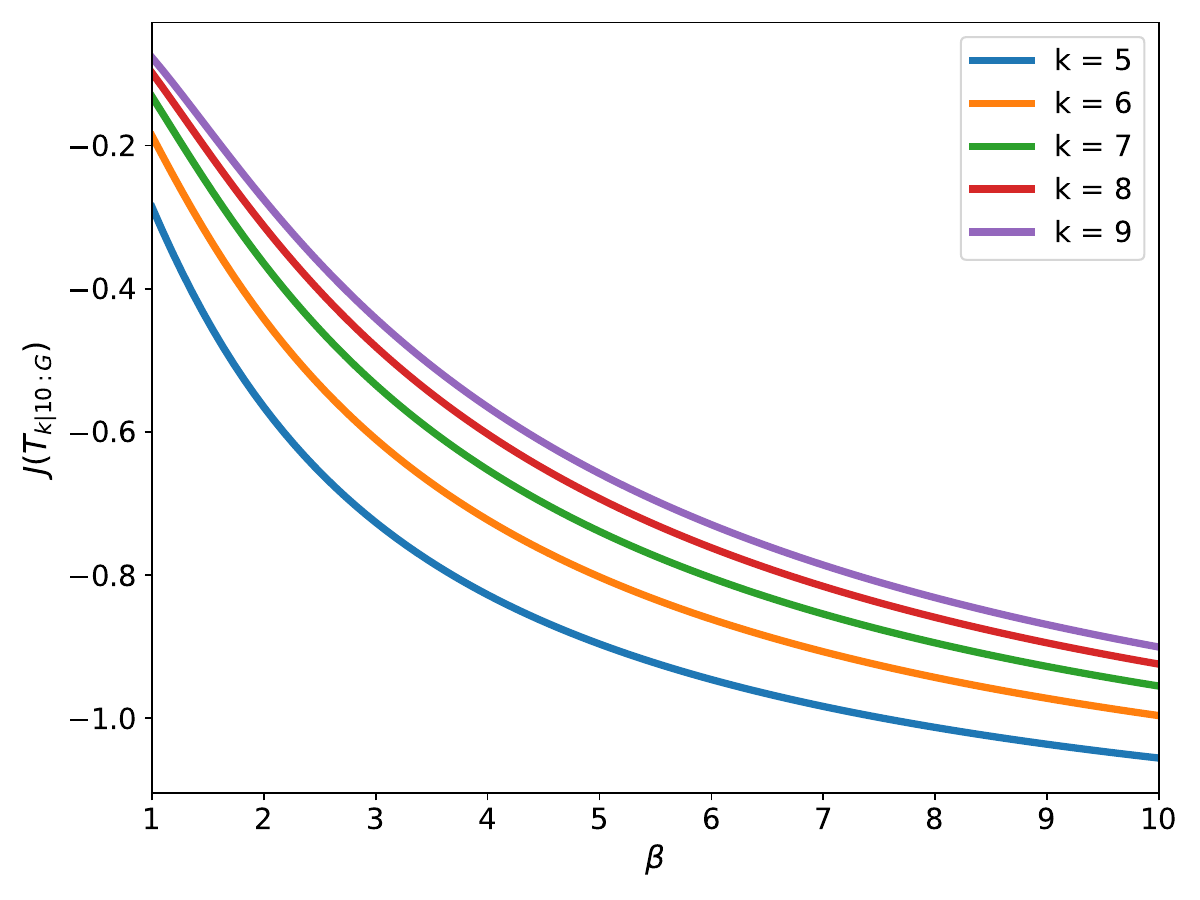}
        \subcaption{CREx versus the scale parameter $\eta$ for fixed $\beta=2.5$.}
    \end{minipage}
    \hspace{0.5cm}
    \begin{minipage}{0.45\textwidth}
        \centering
        \includegraphics[width=7.5cm,height=6cm]{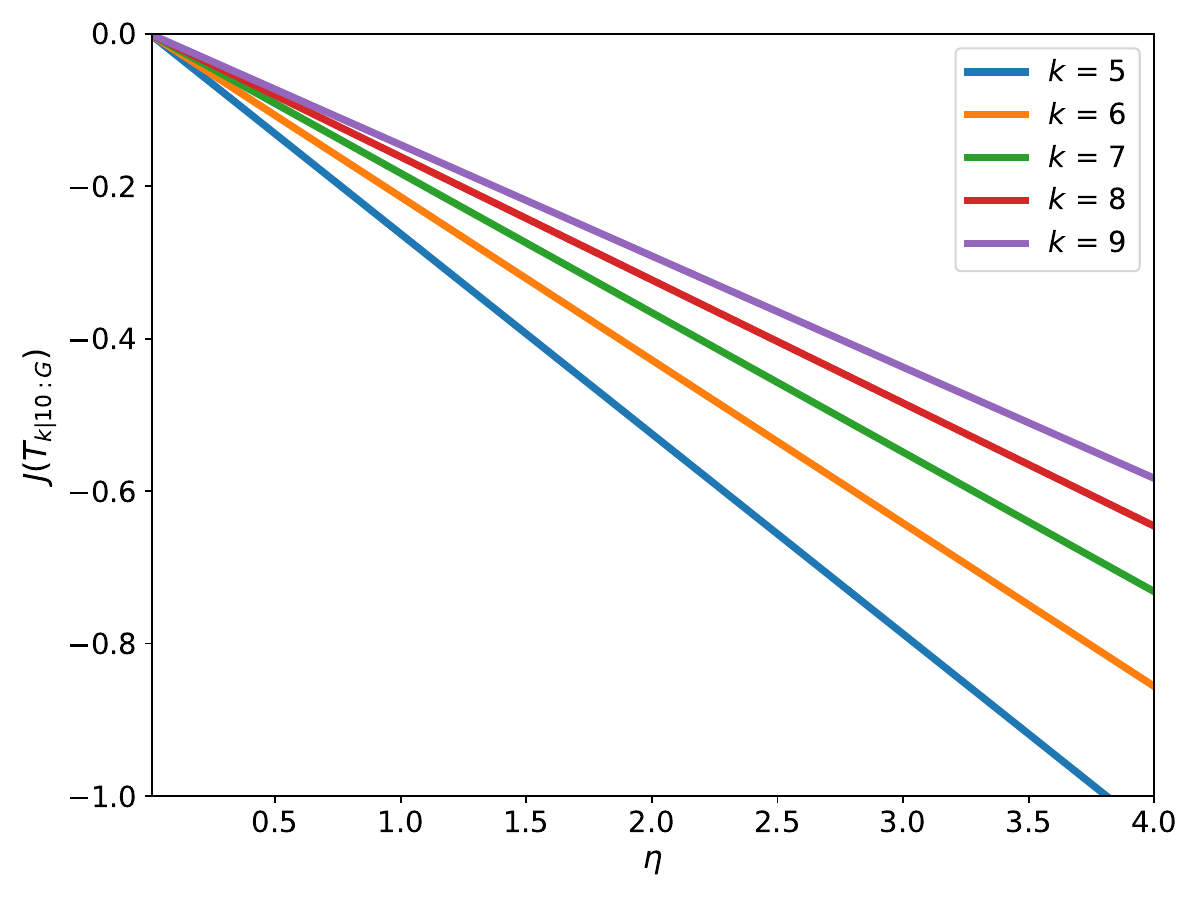}
        \subcaption{CREx versus the shape parameter $\beta$ for fixed $\eta=2.5$.}
    \end{minipage}
    \caption{The plots of CREx for $\mathcal C(k|n{:}G)$ systems under the Weibull distribution.}\label{fig1}
\end{figure}
\end{exe}
Now, we recall some stochastic orders, which will be used to obtain bounds and characterizations. For more details, one may refer to \cite{shaked2007}. Let $X$ and $Y$ be RVs with PDFs $f_{X},f_{X}$ and SFs $\bar F_{X},\bar F_{Y}$, respectively.
\begin{enumerate}
    \item \textbf{Stochastic Order.}  
     $X_1$ is said to be smaller than $X_2$ in the usual stochastic order, denoted by
    $X \leq_{st} X$, whenever $\bar F_{X}(w)\leq \bar F_{Y}(w), ~\forall w.$
    \item \textbf{Hazard Rate Order.} 
   $X$ is said to be smaller than $Y$ in the hazard rate order, written as
    $X \leq_{hr} Y$, if $ \bar h_X(t)\geq \bar h_Y(t), ~ t\geq 0,$
    or, equivalently, if $\frac{\bar F_{Y}(t)}{\bar F_{X}(t)}$
    is an increasing function of $t$. 
    The hazard rate function is defined by $\bar h_X(t)=\frac{f_{X}(t)}{\bar F_{X}(t)}$.

    \item \textbf{Dispersive Order.} 
     $X$ is said to be smaller than $Y$ in the dispersive order, denoted by
    $X \leq_{disp} Y$, if
$F_{X}^{-1}(u_1)-F_{X}^{-1}(u_2)\geq  F_{Y}^{-1}(u_1)-F_{Y}^{-1}(u_2),
        ~~ 0<u_1<u_2<1.$
\end{enumerate}
\cite{jew1989} proposed the location-independent riskier (lir) order, motivated by considerations arising in expected utility theory and actuarial problems. It is a well-established fact that
\begin{align*}
    X\leq_{disp}Y
    \iff 
    f_{Y}\!\left(F_{Y}^{-1}(u)\right)
    \leq
    f_{X}\!\left(F_{X}^{-1}(u)\right),~ 0<u<1.
\end{align*}
Moreover, the following chain of implications holds:
\begin{align}\label{eq2.3}
    \text{If either $X$ or $Y$ is DFR, then }
    X\leq_{disp}Y
    \implies
    X\leq_{lir}Y.
\end{align}

Combining relation \eqref{eq1.2} with \eqref{eq2.3}, one immediately obtains that
\begin{align}\label{eq::2.4}
    X\leq_{disp}Y
    \quad \Longrightarrow \quad
    J(X)\geq J(Y).
\end{align}

Now, let $Z$ be a nonnegative RV with CDF $\bar F_{Z}$, and define
\begin{align*}
    \eta_Z(w)
    =
    \int_0^w \bar F_Z(t)\,dt,~~ w>0.
\end{align*}
\cite{land1994} established the following characterization:
\begin{align}\label{eq2.4}
    X\leq_{lir}Y
    \iff
    \eta_Y^{-1}(w)-\eta_X^{-1}(w)
    \text{ is increasing in } w>0.
\end{align}

Before proceeding to the next theorem, we now establish that the proposed measure  reverses the hazard rate ordering under the DFR condition.
\begin{c1}
Suppose that $X\leq_{hr}Y$ and assume further that either $X$ or $Y$ possesses the DFR property. Then, $J(X)\geq J(Y).$
\end{c1}
\begin{proof}
If $X\leq_{hr}Y$, then $\bar h_X(w)\geq \bar h_Y(w),~ w\geq0,$
where $\bar h_X(w)=f_X(w)/\bar F_X(w)$ is the hazard rate function. Hence, $\int_0^w \bar h_X(t)\,dt
\geq
\int_0^w \bar h_Y(t)\,dt .$
Since $\bar F(w)=\exp\!\left(-\int_0^w \bar h(t)\,dt\right),$
it follows that $\bar F_X(w)\leq \bar F_Y(w), ~w\geq0.$
Therefore, $J(X)\geq J(Y).$
Hence the proof.
\end{proof}
\begin{theorem}
For a DFR RV $X$, the inequality
 $J(X_{1:k})
    \geq
    J(T_{k|n:G})$ holds whenever $2k\geq n$.
\end{theorem}
\begin{proof}
Since $X$ is assumed to be DFR, the series system lifetime $X_{1:k}$ also inherits the DFR property. In addition, Theorem 4.5 of \cite{ery2012} asserts that
$X_{1:k}\leq_{hr}T_{k|n:G},
    ~ 2k\geq n.$
Therefore, by the preceding corollary,
$J(X_{1:k})
    \geq
    J(T_{k|n:G}).$
Hence, the desired result follows.
\end{proof}
\begin{theorem}
Let $X\leq_{lir}Y$. Further, assume that
\begin{align*}
    \Phi(t):=
    \frac{\Phi(\bar G_{k|n:G}(s))}{s},~0<s<1,
\end{align*}
is decreasing in $s$. Then, $J(T_{k|n:G}^X)\leq J(T_{k|n:G}^Y).$
\end{theorem}
\begin{proof}
Recall that
\begin{align*}
    \bar G_{k|n:G}(v)
    =
    (n-k+1)(1-v)^k-(n-k)(1-v)^{k+1},
    \qquad 0<v<1.
\end{align*}
Hence, the SF  can be expressed as
$S^X_{k|n:G}(w)
    =
    \bar G_{k|n:G}(F_X(w)).$ Since $X\leq_{lir}Y$, relation \eqref{eq2.4} yields $\eta_Y^{-1}(w)-\eta_X^{-1}(w)$
is increasing in $w>0$. Differentiating both sides gives
\begin{align*}
    \frac{d}{dw}
    \left[
    \eta_Y^{-1}(w)-\eta_X^{-1}(w)
    \right]
    =
    \frac{1}{F_Y(\eta_Y^{-1}(w))}
    -
    \frac{1}{F_X(\eta_X^{-1}(w))}
    \geq 0,
\end{align*}
which implies
\begin{align}\label{eq:new1}
    F_X(w)
    \geq
    F_Y\!\left(
    \eta_Y^{-1}(\eta_X(w))
    \right),~ w>0.
\end{align}

Now, by the definition of CREx,
\begin{align}
    J(T^X_{k|n:G})
    &=
    -\frac12
    \int_0^\infty
    \left(
    S^X_{k|n:G}(w)
    \right)^2
    dw
    \nonumber\\
    &=-\frac12
    \int_0^\infty
    \Phi\!\left(
    \bar G_{k|n:G}(F_X(w))
    \right)
    dw
    \nonumber\\
    &=-\frac12
    \int_0^\infty
    \frac{
    \Phi\!\left(
    \bar G_{k|n:G}(F_X(w))
    \right)
    }{F_X(w)}
    F_X(w)\,dw.
    \label{eq:new2}
\end{align}
Since $\frac{\Phi(\bar G_{k|n:G}(s))}{s}$
is decreasing in $s$, together with \eqref{eq:new1}, we obtain
\begin{align}
    \frac{
    \Phi(\bar G_{k|n:G}(F_X(w)))
    }{F_X(w)}
    \leq
    \frac{
    \Phi\!\left(
    \bar G_{k|n:G}
    \left(
    F_Y(\eta_Y^{-1}(\eta_X(w)))
    \right)
    \right)
    }{
    F_Y(\eta_Y^{-1}(\eta_X(w)))
    }.
    \label{eq:new3}
\end{align}

Thus, from \eqref{eq:new3} and \eqref{eq:new2}, we have
\begin{align}
    J(T^X_{k|n:G})
    \geq -\frac12
    \int_0^\infty
    \frac{
    \Phi\!\left(
    \bar G_{k|n:G}
    \left(
    F_Y(\eta_Y^{-1}(\eta_X(w)))
    \right)
    \right)
    }{
    F_Y(\eta_Y^{-1}(\eta_X(w)))
    }
    F_X(w)\,dw.
    \label{eq:new4}
\end{align}
Now, perform the transformation $v=\eta_Y^{-1}(\eta_X(w)).$
Then, 
$\eta_Y(v)=\eta_X(w).$
Differentiating both sides yields, 
$dw
    =
    \frac{F_Y(v)}{F_X(w)}\,dv.$

Therefore, \eqref{eq:new4} becomes
\begin{align*}
    J(T^X_{k|n:G})
    &\geq -\frac12
    \int_{\eta_Y^{-1}(\eta_X(0))}^{\infty}
    \Phi\!\left(
    \bar G_{k|n:G}(F_Y(v))
    \right)\,dv
    \\
    &=
    -\frac12
    \int_0^\infty
    \left(
    S^Y_{k|n:G}(v)
    \right)^2
    dv=
    J(T^Y_{k|n:G}).
\end{align*}
This completes the proof.
\end{proof}
\begin{theorem}\label{th2.3}
    For $2k\geq n$, the following bounds hold for the CREx of $T_{k|n:G}$:
    \begin{align*}
        \mathfrak{b}_1 J(X)\leq J(T_{k|n:G})\leq \mathfrak b_2
J(X),    \end{align*}
where $\mathfrak b_1=\sup_{v\in(0,1)}\frac{\Phi\left(\bar G_{k|n:G}(v)\right)}{\Phi(1-v)}$ and $\mathfrak b_2=\inf_{v\in(0,1)}\frac{\Phi\left(\bar G_{k|n:G}(v)\right)}{\Phi(1-v)}$, provided that $\inf_v$ and $\sup_v$ exist.
\end{theorem}
\begin{proof}
   \begin{align*}
        J(T_{k|n:G})=-\frac12\int_0^1 \frac{\Phi\left(\bar G_{k|n:G}(v)\right)}{f(F^{-1}(v))}\,dv&=-\frac12\int_0^1 \frac{\Phi(\bar G_{k|n{:}G}(v))}{\Phi(1-v)}\frac{\Phi(1-v)}{f(F^{-1}(v))}\,dv\\
        &\geq -\frac12\sup_{v\in(0,1)}\frac{\Phi(\bar G_{k|n{:}G}(v))}{\Phi(1-v)}\int_0^1 \frac{\Phi(1-v)}{f(F^{-1}(v))}\,dv=\mathfrak b_2 J(X).
   \end{align*}
  Proceeding in the same manner, the upper bound is established.
\end{proof}
For the purpose of illustration, we present the following example pertaining to Theorem \ref{th2.3}.     
\begin{exe}
  Let $T_{k|10{:}G}$ characterizes the lifetime of a consecutive $\mathcal C(k|10{:}G)$ system  for $k=5,...,10$, and suppose that each component is distributed according to the Lomax distribution with CDF given by
    \begin{align*}
        F(w)=1-\left(1+\frac{w}{\lambda}\right)^{-\alpha},~w>0,~\beta>0,~\eta>0
    \end{align*}
    where $\beta$ and $\eta$ represent the shape and scale parameters, respectively. We choose $\beta=2$ and $\eta=3$ and plot $\mathfrak b_1\cdot J(X)$, $J(T_{k|n{:}G})$, and $\mathfrak b_2\cdot J(X)$ in Figure \ref{fig2}.
    \begin{figure}[H]
        \centering
        \includegraphics[width=14cm, height=7cm]{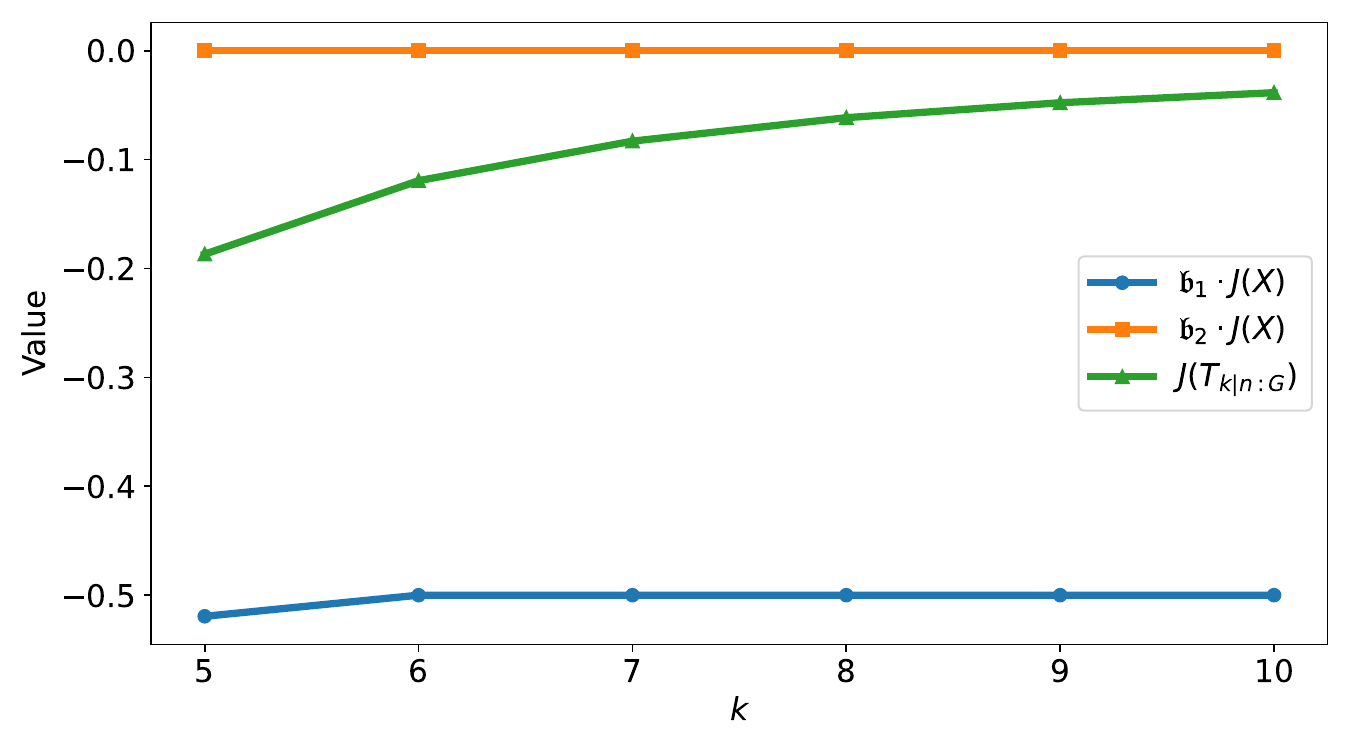}
        \caption{Graphical illustration of $\mathfrak b_1\cdot J(X)$, $J(T_{k|n{:}G})$, and $\mathfrak b_2\cdot J(X)$ for $k=5,6,...,10$.}
        \label{fig2}
    \end{figure}
\end{exe}
The following result establishes two new bounds for the CREx measure of consecutive systems.
\begin{theorem}
Assume that $X$ is an absolutely continuous RV with support set $\mathcal A$, and denote $\mathcal M_1=\inf_{w\in \mathcal A} f(w)$ and $\mathcal M_2=\sup_{w\in \mathcal A} f(w),$ where $f$ is PDF of $X$. Then, the CREx  satisfies
\begin{align*}
    \frac{J(V_{k|n:G})}{\mathcal M_2}
    \leq
    J(T_{k|n:G})
    \leq
    \frac{J(V_{k|n:G})}{\mathcal M_1},
\end{align*}
where
$J(V_{k|n:G})
    =
    -\frac12
    \int_0^1
    \Phi\!\left(
    \bar G_{k|n:G}(v)
    \right)\,dv$ and $V_{k|n:G}$ represents the lifetime of a  $\mathcal C(k|n{:}G)$ system, where the $n$ component lifetimes are iid uniform RVs on the interval $[0,1].$
\end{theorem}
\begin{proof}
    Since $\mathcal M_1\leq f(F^{-1}(v))\leq \mathcal M_2$, $0<v<1$. Using \eqref{eq2.3}, we get
    \begin{align*}
        J(T_{k|n:G})=-\frac{1}{2}\int_0^1 \frac{\Phi(\bar G_{k|n:G}(v))}{f(F^{-1}(v))}\,dv\leq -\frac{1}{2\mathcal M_2}\int_0^1 \Phi(\bar G_{k|n:G}(v))\,dv.
    \end{align*}
   An analogous argument yields the lower bound.
    \end{proof}
 \begin{theorem}\label{th2.5}
Assume that $T_{k|n:G}^{X}$ and $T_{k|n:G}^{Y}$ represent the lifetimes of two $\mathcal C(k|n{:}G)$ systems generated from absolutely continuous RVs $X$ and $Y$, having distribution functions $F_X$ and $F_Y$, and corresponding PDFs $f_X$ and $f_Y$, respectively. Assume that $X \leq_{\mathrm{disp}} Y.$
Then, $F_X$ and $F_Y$ differ only by a location parameter if and only if ${J}\!\left(T_{k|n:G}^{X}\right)
=
{J}\!\left(T_{k|n:G}^{Y}\right),$
for every pair $(k,n)$ satisfying $2k\ge n$.
\end{theorem}
\begin{proof}
Assume that CDFs $F_X$ and $F_Y$ are members of a common location family. Then, for some $d\in\mathbb{R}$,
$
F_Y(w)=F_X(w-d).
$
Consequently,
$
f_Y(w)=f_X(w-d)
$
and
$
F_Y^{-1}(v)=F_X^{-1}(v)+d.
$
Hence,
$$
f_Y(F_Y^{-1}(v))
=
f_X(F_X^{-1}(v)).
$$
Using representation \eqref{eq2.3}, we obtain
$$
J(T_{k|n:G}^{Y})
=
-\frac12\int_0^1
\frac{\Phi(\overline G_{k|n:G}(v))}
{f_Y(F_Y^{-1}(v))}\,dv
=-\frac12
\int_0^1
\frac{\Phi(\overline G_{k|n:G}(v))}
{f_X(F_X^{-1}(v))}\,dv
=
{J}(T_{k|n:G}^{X}).
$$

Therefore,
$
{J}(T_{k|n:G}^{X})
=
{J}(T_{k|n:G}^{Y}).
$

Conversely, assume that 
$
{J}(T_{k|n:G}^{X})
=
{J}(T_{k|n:G}^{Y}).
$ Then, we get
$$
-\frac12\int_{0}^{1}
\Phi(\overline G_{k|n:G}(v))
\left[
\frac{1}{f_X(F_X^{-1}(v))}
-
\frac{1}{f_Y(F_Y^{-1}(v))}
\right]dv
=0.
$$
Since
$
\Phi(\overline G_{k|n:G}(v))>0
$
for $0<v<1$, and $X\leq_{\mathrm{disp}}Y$, it follows from the characterization of the dispersive order that
\begin{align*}
    f_X(F_X^{-1}(v))
\ge
f_Y(F_Y^{-1}(v))\implies \frac{1}{f_X(F_X^{-1}(v))}
-
\frac{1}{f_Y(F_Y^{-1}(v))}
\le 0.
\end{align*}
Therefore, the integrand is non-positive on $(0,1)$. Since its weighted integral is zero and $\Phi(\bar G_{k|n:G}(v))$ is positive, we must have
$$
\frac{1}{f_X(F_X^{-1}(v))}
=
\frac{1}{f_Y(F_Y^{-1}(v))}
$$
for every $u\in(0,1)$. Equivalently,
$$
f_X(F_X^{-1}(v))
=
f_Y(F_Y^{-1}(v)).
$$

Thus,
$
(F_X^{-1})'(v)
=
(F_Y^{-1})'(v),
$
which yields
$
F_X^{-1}(v)=F_Y^{-1}(v)+d
$
for some $d\in\mathbb{R}$. Hence, $F_X$ and $F_Y$ differ only by a location shift.
\end{proof}
\begin{theorem}
Suppose that the assumptions of Theorem \ref{th2.5} hold. Then, the CDFs $F_X$ and $F_Y$ belong to a common location--scale family
if and only if
\begin{equation*}
\frac{J(T_{k|n:G}^{X})}{J(Y)}
=
\frac{J(T_{k|n:G}^{Y})}{J(X)},
\end{equation*}
for all $k$ and $n$ satisfying $2k\ge n$, provided that $J(X),J(Y)<0$.
\end{theorem}

\begin{proof}
The necessity part is straightforward, and hence, it is enough to establish the sufficiency part.

Assume that
\begin{equation}\label{eqA1}
\frac{J(T_{k|n:G}^{X})}{J(Y)}
=
\frac{J(T_{k|n:G}^{Y})}{J(X)}.
\end{equation}

Using the representation of extropy corresponding to $\Phi(v)=v^2$, we obtain
$$
\frac{J(T_{k|n:G}^{X})}{J(Y)}
=
-\frac12
\int_0^1
\frac{\Phi(\overline G_{k|n:G}(v))}
{J(Y)\,f_X(F_X^{-1}(v))}\,dv,
$$
and similarly,
$$
\frac{J(T_{k|n:G}^{Y})}{J(X)}
=
-\frac12
\int_0^1
\frac{\Phi(\overline G_{k|n:G}(v))}
{J(X)\,f_Y(F_Y^{-1}(v))}\,dv.
$$

Substituting these expressions into \eqref{eqA1} yields
$$
-\frac12
\int_0^1
\Phi(\overline G_{k|n:G}(v))
\left[
\frac{1}{J(Y)\,f_X(F_X^{-1}(v))}
-
\frac{1}{J(X)\,f_Y(F_Y^{-1}(v))}
\right]dv
=0.
$$

Now, let
$
c=\dfrac{J(Y)}{J(X)}.
$
Since $X\leq_{\mathrm{disp}}Y$, Equation \eqref{eq::2.4} implies
$
J(X)\ge J(Y),
$
and therefore $c\ge1$. Consequently, the above equation becomes
$$
-\frac{1}{2J(X)}
\int_0^1
\Phi(\overline G_{k|n:G}(v))
\left[
\frac{1}{cf_X(F_X^{-1}(v))}
-
\frac{1}{f_Y(F_Y^{-1}(v))}
\right]dv
=0.
$$

Since $X\leq_{disp}Y$, Equation \eqref{eq::2.4} further gives
$$
f_X(F_X^{-1}(v))
\ge
f_Y(F_Y^{-1}(v)),
\qquad 0<v<1.
$$
Since $c\ge1$, it follows that
$$
cf_X(F_X^{-1}(v))
\ge
f_Y(F_Y^{-1}(v)),
$$
and hence
$$
\frac1{cf_X(F_X^{-1}(v))}
-
\frac1{f_Y(F_Y^{-1}(v))}
\leq 0,
\qquad 0<v<1.
$$

Moreover,
$$
\Phi(\overline G_{k|n:G}(v))
=
\overline G_{k|n:G}^{\,2}(v)>0,
\qquad 0<v<1.
$$
Therefore, the integrand is non-positive on $(0,1)$. Since its integral vanishes, we necessarily have
$$
\frac{1}{cf_X(F_X^{-1}(v))}
=
\frac{1}{f_Y(F_Y^{-1}(v))}
\quad \text{for every } v\in(0,1).
$$

Using the identity $(F^{-1})'(v)=\frac1{f(F^{-1}(v))},$ we obtain
$$
(F_X^{-1})'(v)
=
c\,(F_Y^{-1})'(v).
$$
Integrating both sides with respect to $v$, it follows that
$$
F_X^{-1}(v)
=
c\,F_Y^{-1}(v)+d,
$$
for some constant $d\in\mathbb{R}$. Equivalently,
$
F_X(w)
=
F_Y\!\left(\frac{w-d}{c}\right),~ c>0,
$
which proves that $F_X$ and $F_Y$ belong to the same location--scale family.
\end{proof}
\section{Dynamic CREx for \texorpdfstring{$\mathcal{C}(k|n{:}G)$}~~\text{system}}
In reliability analysis, the uncertainty associated with the residual lifetime of a coherent system plays an important role in studying the ageing behavior and future performance of the system. Let \(T_{k|n:G}\) denote the lifetime of a $\mathcal C(k|n{:}G)$ system with common CDF \(F\), SF \(\bar F=1-F\), and PDF \(f\). Further, let \(T_{1:n},T_{2:n},\ldots,T_{n:n}\) denote the ordered component lifetimes.  

We consider the conditional RV $(T_{k|n:G}-s \mid T_{1:n}>s),$ which represents the residual lifetime of the $\mathcal C(k|n{:}G)$ system at time \(s\), given that all the components of the system are still functioning at time \(s\). Equivalently, the condition \(T_{1:n}>s\) implies that the earliest component failure has not yet occurred up to time \(s\). Thus, the above conditional variable describes the remaining lifetime of the system under the assumption that the system is completely operational at age \(s\). 

\cite{ery2010} obtained the mean residual life (MRL) of a $\mathcal C(k|n{:}G)$ system as
\begin{align*}
E(T_{k|n:G}-s\mid T_{1:n}>s)
&=
(n-k+1)
\int_0^\infty
\left(
\frac{\bar F(w+s)}{\bar F(s)}
\right)^k
dw
-(n-k)
\int_0^\infty
\left(
\frac{\bar F(w+s)}{\bar F(s)}
\right)^{k+1}
dw,
\end{align*}
for $2k\ge n$.

Motivated by the above MRL expression for $\mathcal C(k|n{:}G)$, in the following, we provide an explicit representation for the DCREx of the residual lifetime of the $\mathcal C(k|n{:}G)$ system.
\begin{theorem}
Let $T_{k|n:G}$ denote the lifetime of the $\mathcal C(k|n{:}G)$ system with SF $\bar F$, and assume that $2k\ge n$. Then, the DCREx is given by
\begin{align*}
J(Y_s)
=
-\frac{\bar F(s)}{2}
\int_0^1
\frac{
\left[
(n-k+1)u^k-(n-k)u^{k+1}
\right]^2
}{
f(F^{-1}(1-u\bar F(s)))
}
\,du,
\end{align*}
where $Y_s=(T_{k|n:G}-s\mid T_{1:n}>s).$
\end{theorem}

\begin{proof}
For $2k\ge n$, similar to \cite{ery2010} (Corollary 1), the SF $\bar F_{Y_s}$ of $Y_s=(T_{k|n:G}-s\mid T_{1:n}>s)$ is defined as
\begin{align*}
\bar F_{Y_s}(w)
=
(n-k+1)
\left(
\frac{\bar F(s+w)}{\bar F(s)}
\right)^k
-
(n-k)
\left(
\frac{\bar F(s+w)}{\bar F(s)}
\right)^{k+1}.
\end{align*}
By the definition of CREx,
\begin{align*}
J(Y_s)
&=
-\frac12
\int_0^\infty
\bar F_{Y_s}^2(w)\,dw\\
&=
-\frac12
\int_0^\infty
\Bigg[
(n-k+1)
\left(
\frac{\bar F(s+w)}{\bar F(s)}
\right)^k
-(n-k)
\left(
\frac{\bar F(s+w)}{\bar F(s)}
\right)^{k+1}
\Bigg]^2
dw.
\end{align*}
Now use the transformation
$
x=\dfrac{\bar F(s+w)}{\bar F(s)}.$ Then
$
dx=-\dfrac{f(s+w)}{\bar F(s)}\,dw,
$
so that

$
dw=
-\dfrac{\bar F(s)}
{f(F^{-1}(1-x\bar F(s)))}
\,dx.$ Further, when $w=0$, we have $x=1$, and as $w\to\infty$, $x\to0$. Therefore,
\begin{align*}
J(Y_s)
=
-\frac{\bar F(s)}{2}
\int_0^1
\frac{
\left[
(n-k+1)x^k-(n-k)x^{k+1}
\right]^2
}{
f(F^{-1}(1-x\bar F(s)))
}
\,dx.
\end{align*}
This completes the proof.
\end{proof}
\begin{exe}
Consider the lifetime $Y_s=(T_{k|n:G}-s \mid T_{1:n}>s),$
where \(T_{k|n:G}\) denotes the lifetime of a $\mathcal C(k|n{:}G)$ system and each component lifetime follows the exponential distribution with CDF $F(w)=1-e^{-\gamma w},~ w>0,\ \gamma>0.$
Since the analytical expression of DCREx for \(Y_t\) is difficult to study explicitly, we present a numerical illustration for \(n=10\) and \(k=5,\ldots,10\). The corresponding plots are displayed in Figure~\ref{fig3}.
\begin{figure}[H]
    \centering
    \includegraphics[width=15cm, height=7.5cm]{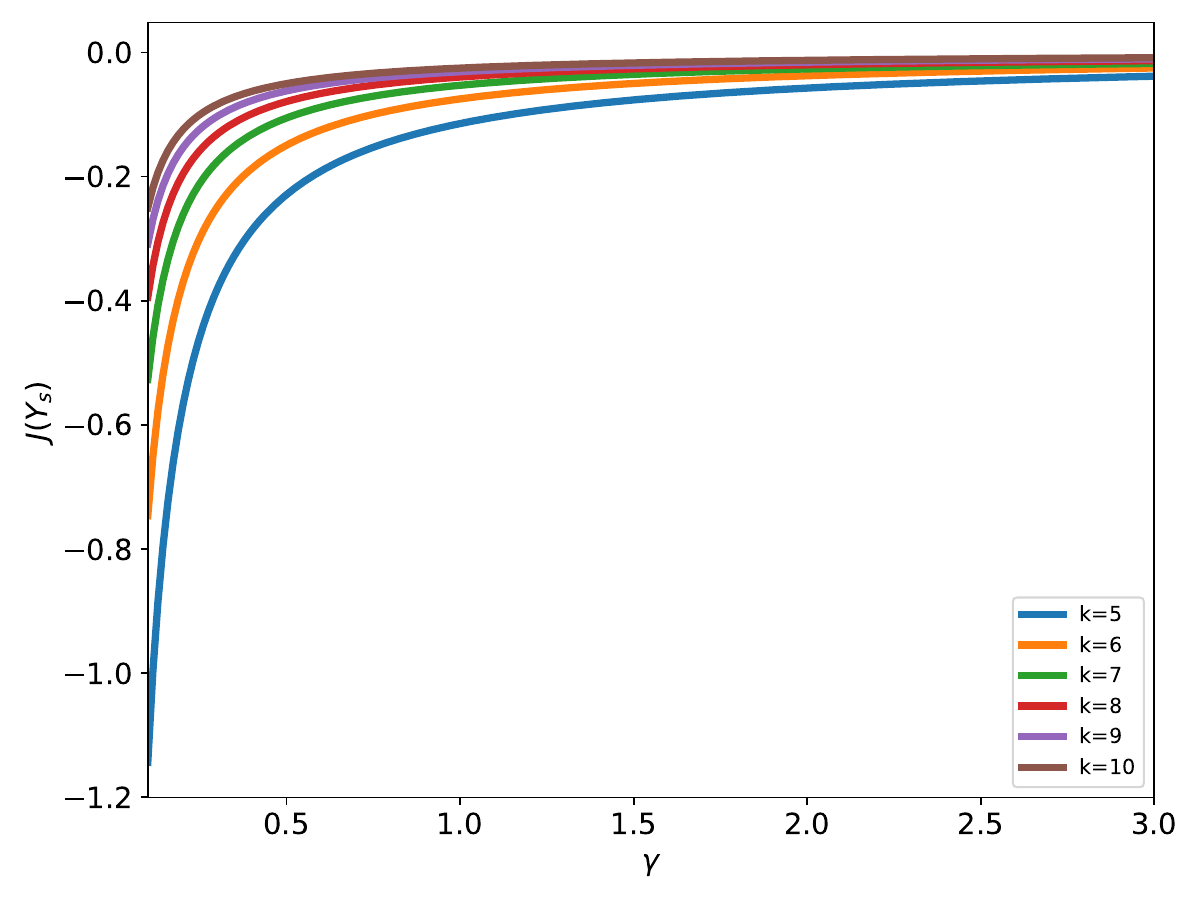}
    \caption{The plot of  DCREx
for a $\mathcal C(k|10{:}G)$ system when component lifetimes follow the exponential distribution with parameter \(\gamma\in(0,3)\), for \(k=5,\ldots,10\).}
    \label{fig3}
\end{figure}
It is evident from Figure~\ref{fig3} that, as the parameter \(\gamma\) increases, the DCREx of the system increases and approaches to zero. This behavior indicates that larger values of \(\gamma\) reduce the uncertainty associated with the residual lifetime of the system. In reliability interpretation, values of DCREx closer to zero correspond to lower uncertainty and greater predictability in system performance and reliability.
\end{exe}

In the following theorem, we establish upper and lower bounds for DCREx measure.
\begin{theorem}
For $2k\geq n$, the DCREx of $(T_{k|n:G}-x|T_{1:n}>s)$ is bounded as follows:
\begin{align*}
    \mathcal M_1\cdot J(s)\leq J(T^s_{k|n:G})\leq \mathcal M_2\cdot J(s),
\end{align*}
where $\sup_{u\in(0,1)}\frac{\Phi(R_{k|n:G}(u))
}{\Phi(u)}=\mathcal M_2$,  $\inf_{u\in(0,1)}\frac{\Phi(R_{k|n:G}(u))
}{\Phi(u)}=\mathcal M_1$, and $\Phi(w)=-\frac{1}{2}w^2$.
\end{theorem}
\begin{proof}
Let
\[
d\mu_s(u)
=
\bar F(s)\,
\frac{\Phi(u)}
     {f(F^{-1}(1-u\bar F(s)))}
\,du.
\]
Then
\[
J_{k|n:G}(s)
=
\int_0^1
\frac{\Phi(R_{k|n:G}(u))}
     {\Phi(u)}
\,d\mu_s(u).
\]

Since
\[
\mathcal M_1
\le
\frac{\Phi(R_{k|n:G}(u))}
     {\Phi(u)}
\le
\mathcal M_2,
\qquad 0<u<1,
\]
we obtain
\[
\mathcal M_1
\int_0^1 d\mu_s(u)
\le
J_{k|n:G}(s)
\le
\mathcal M_2
\int_0^1 d\mu_s(u).
\]
Noting that $\int_0^1 d\mu_s(u)
=
J(s),$
it follows that
\[
\mathcal M_1\,J(s)
\le
J_{k|n:G}(s)
\le
\mathcal M_2\,J(s).
\]
This completes the proof.
\end{proof}
\section{Nonparametric estimation}
In this section, we develop a distribution--free estimator for CREx associated with consecutive $\mathcal{C}(k|n:G)$ systems. Let $X_1,X_2,\ldots,X_N$
be a random sample of size $N$ drawn from an absolutely continuous non--negative distribution function $F$. Denote the corresponding order statistics by $X_{1:N}\leq X_{2:N}\leq \cdots \leq X_{N:N}.$
It is important to emphasize that these observations are external sample lifetimes and are not the actual component lifetimes constituting the considered  $\mathcal{C}(k|n:G)$ system. The observed sample is used to construct an empirical approximation of the underlying reliability structure, which in turn serves as the basis for estimating the CREx.

Using representation \eqref{eq:2.1}, the CREx of $T_{k|n:G}$ can be expressed as
$$
J(T_{k|n:G})
=-\frac12
\int_0^1
\frac{\Phi\!\left(\overline G_{k|n:G}(v)\right)}
{f(F^{-1}(v))}
\,dv
=-\frac12
\int_0^1
\Phi\!\left(\overline G_{k|n:G}(v)\right)
\left[\frac{dF^{-1}(v)}{dv}\right]
dv,~~2k\geq n.
$$

Motivated by the above formulation, we estimate
$J(T_{k|n:G})$
through a nonparametric approximation of the derivative of the quantile function. Following the methodology introduced by \cite{vas1976}, the derivative $\frac{dF^{-1}(v)}{dv}$
is approximated by a local finite--difference slope computed from adjacent order statistics. Specifically,
$$
\frac{dF^{-1}(v)}{dv}
\approx
\frac{N\bigl(X_{i+m:N}-X_{i-m:N}\bigr)}{2m},
$$
where the boundary conventions
$X_{i:N}=X_{1:N}, \quad i<1,$
and $X_{i:N}=X_{N:N}, \quad i>N,$
are adopted. Here, $m$ denotes a positive integer window parameter satisfying $m\leq \frac{N}{2}.$

Consequently, an empirical estimator for
$J(T_{k|n:G})$
is obtained as follows:
\begin{align}\label{eq3.1}
    \hat J_1(T_{k|n:G})&=-\frac{1}{2N}\sum_{i=1}^N\Phi\left(\bar G_{k|n:G}\left(\frac{i}{N+1}\right)\right)\left(\frac{N\bigl(X_{i+m:N}-X_{i-m:N}\bigr)}{2m}\right)\nonumber\\
    &=-\frac{1}{2N}\sum_{i=1}^N\Phi\left((n-k+1)\left(1-\frac{i}{N+1}\right)^k-(n-k)\left(1-\frac{i}{N+1}\right)^{k+1}\right)\nonumber\\&~~~~~~~~~~~\times\left(\frac{N\bigl(X_{i+m:N}-X_{i-m:N}\bigr)}{2m}\right)
\end{align}.

Herein, we prove the consistency of the proposed estimator.
\begin{theorem}
Assume $T_{k|n:G}$ represents the lifetime of a  $\mathcal C(k|n{:}G)$ system. For fixed integers $k$ and $n$, the estimator $\widehat J(T_{k|n:G})$
satisfies $\widehat J(T_{k|n:G})
\xrightarrow{\text{a.s.}} J(T_{k|n:G})~ \text{as}~ N\to\infty,$
provided that $m\to\infty
~\text{and}~
\frac{m}{N}\to0.$
Hence, $\widehat{J}(T_{k|n:G})$ is a strongly consistent estimator of
$J(T_{k|n:G})$.
\end{theorem}

\begin{proof}
To prove the result, we employ the asymptotic argument used in \cite{nou2011}. Under the conditions $m\to\infty
~\text{and}~
\frac{m}{N}\to0,$
the finite--difference quantity associated with the order statistics provides an approximation to PDF. In particular,
\begin{align*}\
    \frac{2m}
{N\bigl(X_{i+m:N}-X_{i-m:N}\bigr)}
&=
\frac{
F_N(X_{i+m:N})-F_N(X_{i-m:N})
}
{
X_{i+m:N}-X_{i-m:N}
}\nonumber\\
&\approx
\frac{
F(X_{i+m:N})-F(X_{i-m:N})
}
{
X_{i+m:N}-X_{i-m:N}
}\\
&\approx\frac{f_{X_{i+m:N}+f_{X_{i-m:N}}}}{2}\approx f(X_{i:N}).
\end{align*} 
Now, using the identity $F_N(X_{i:N})=\frac{i}{N+1},$
the proposed estimator can be written as
$$
\widehat{ J}(T_{k|n:G})
=-\frac12\frac1N
\sum_{i=1}^N
\Phi\!\left(
\overline G_{k|n:G}\!\left(\frac{i}{N+1}\right)
\right)
\left(
\frac{
N\bigl(X_{i+m:N}-X_{i-m:N}\bigr)
}
{2m}
\right).
$$

Using the previous approximation,
$$
\frac{
N\bigl(X_{i+m:N}-X_{i-m:N}\bigr)
}
{2m}
\approx
\frac1{f(X_{i:N})},
$$
which gives
\begin{align*}
    \widehat{J}(T_{k|n:G})
\approx-\frac12
\frac1N
\sum_{i=1}^N
\frac{
\Phi\!\left(
\overline G_{k|n:G}\!\left(
F_N(X_{i:N})
\right)
\right)
}
{f(X_{i:N})}.
\end{align*}
Since $F_N(X_{i:N})
\xrightarrow{\text{a.s.}}
F(X_{i:N}),$
and $\overline G_{k|n:G}(\cdot)$ together with $\Phi(\cdot)$ are continuous, we further obtain
$$
\widehat{J}(T_{k|n:G})
\approx-\frac12
\frac1N
\sum_{i=1}^N
\frac{
\Phi\!\left(
\overline G_{k|n:G}\!\left(
F(X_{i:N})
\right)
\right)
}
{f(X_{i:N})}.
$$
Because the ordered observations are merely a permutation of the original sample,
$$
\frac1N
\sum_{i=1}^N
\frac{
\Phi\!\left(
\overline G_{k|n:G}\!\left(
F(X_{i:N})
\right)
\right)
}
{f(X_{i:N})}
=
\frac1N
\sum_{i=1}^N
\frac{
\Phi\!\left(
\overline G_{k|n:G}\!\left(
F(X_i)
\right)
\right)
}
{f(X_i)}.
$$

Finally, by the Strong Law of Large Numbers,
$$
\frac1N
\sum_{i=1}^N
\frac{
\Phi\!\left(
\overline G_{k|n:G}\!\left(
F(X_i)
\right)
\right)
}
{f(X_i)}
\xrightarrow{\text{a.s.}}
\mathbb E\!\left[
\frac{
\Phi\!\left(
\overline G_{k|n:G}(F(X))
\right)
}
{f(X)}
\right].
$$

Using the transformation $u=F(w)$, we have
$$
-\frac{1}{2}\mathbb E\!\left[
\frac{
\Phi\!\left(
\overline G_{k|n:G}(F(X))
\right)
}
{f(X)}
\right]
=
-\frac12\int_0^\infty
\Phi\!\left(
\overline G_{k|n:G}(F(w))
\right)\,dw
=
J(T_{k|n:G}).
$$

Therefore, $\widehat{J}_1(T_{k|n:G})
\xrightarrow{\text{a.s.}}
J(T_{k|n:G}),$
which establishes the strong consistency of the estimator.
\end{proof}
\subsection{Simulation study}
In this subsection, a Monte Carlo simulation study is conducted to examine the finite-sample performance of the proposed estimator of CREx defined in \eqref{eq3.1} for consecutive $\mathcal C(k|n{:}G)$ systems. The component lifetimes are assumed to follow the Lomax distribution with CDF
\begin{equation*}
F(w)=1-\left(1+\frac{w}{\eta}\right)^{-\beta}, \qquad 
w>0,\ \beta>0,\ \eta>0,
\end{equation*}
where $\beta$ and $\eta$ denote the shape and scale parameters, respectively.
For each selected combination of $n$, $k$, and sample size $N$, independent random samples are generated from the Lomax model, and the estimator proposed in \eqref{eq3.1} is computed. The window size used in the spacing approximation is taken as $m=\lfloor \sqrt{N}\rfloor$, which satisfies the regularity conditions $m\to\infty$ and $\frac{m}{N}\to0$ as $N\to\infty$.
The simulation experiment is repeated $5000$ times for every parameter configuration. The performance of the estimator is evaluated in terms of the empirical bias, root mean squared error (RMSE), and the associated $95\%$ confidence interval (CI). The study is carried out for several choices of $k$ and $n$, with sample sizes $N=20,30,40,~\text{and}~50$. The corresponding numerical results are reported in Tables \ref{tab1}, \ref{tab2} and \ref{tab3} for the parameter combinations
$(\alpha,\lambda)=(2.5,1.0)$,
$(\alpha,\lambda)=(3.5,2.5)$, and
$(\alpha,\lambda)=(1.5,4.5)$, respectively.

The simulation results demonstrate that the proposed estimator performs satisfactorily under different system structures and parameter settings. In particular, as the sample size increases, both the  bias and RMSE decrease steadily, indicating the consistency and improved accuracy of the estimator. Furthermore, the associated $95\%$ CIs narrower as the sample size increases, reflecting the increasing precision and stability of the estimation procedure.
\newpage

\begin{sidewaystable}
\centering
\caption{Simulation results for $\widehat J(T_{k|n:G})$ under the Lomax distribution with parameters $\beta=2.5$ and $\eta=1$, showing the Mean Estimate, Bias, RMSE, and 95\% CIs for different values of $n$, $k$, and sample size $N$.}
\label{tab1}

\renewcommand{\arraystretch}{1.5}
\setlength{\tabcolsep}{3pt}

\scriptsize

\begin{tabular}{cccccccccccccccccc}
\toprule

& &
\multicolumn{4}{c}{$N=20$} &
\multicolumn{4}{c}{$N=30$} &
\multicolumn{4}{c}{$N=40$} &
\multicolumn{4}{c}{$N=50$} \\

\cmidrule(lr){3-6}
\cmidrule(lr){7-10}
\cmidrule(lr){11-14}
\cmidrule(lr){15-18}

$n$ & $k$
& Mean & Bias & RMSE & 95\% CI
& Mean & Bias & RMSE & 95\% CI
& Mean & Bias & RMSE & 95\% CI
& Mean & Bias & RMSE & 95\% CI \\

\midrule

5 & 3
& -0.0554 & 0.0076 & 0.0201 & (-0.0558,-0.0551)
& -0.0564 & 0.0066 & 0.0173 & (-0.0567,-0.0561)
& -0.0569 & 0.0062 & 0.0155 & (-0.0572,-0.0566)
& -0.0571 & 0.0059 & 0.0138 & (-0.0574,-0.0569) \\

5 & 4
& -0.0240 & 0.0090 & 0.0129 & (-0.0242,-0.0239)
& -0.0256 & 0.0075 & 0.0110 & (-0.0258,-0.0255)
& -0.0263 & 0.0068 & 0.0099 & (-0.0264,-0.0262)
& -0.0269 & 0.0062 & 0.0091 & (-0.0270,-0.0268) \\

5 & 5
& -0.0127 & 0.0081 & 0.0097 & (-0.0128,-0.0126)
& -0.0140 & 0.0069 & 0.0084 & (-0.0141,-0.0139)
& -0.0146 & 0.0062 & 0.0077 & (-0.0147,-0.0145)
& -0.0151 & 0.0058 & 0.0071 & (-0.0152,-0.0150) \\

\midrule

7 & 5
& -0.0211 & 0.0090 & 0.0123 & (-0.0213,-0.0209)
& -0.0223 & 0.0077 & 0.0107 & (-0.0225,-0.0222)
& -0.0233 & 0.0068 & 0.0095 & (-0.0234,-0.0231)
& -0.0237 & 0.0063 & 0.0089 & (-0.0238,-0.0236) \\

7 & 6
& -0.0121 & 0.0081 & 0.0096 & (-0.0122,-0.0120)
& -0.0133 & 0.0069 & 0.0083 & (-0.0134,-0.0132)
& -0.0141 & 0.0061 & 0.0075 & (-0.0141,-0.0140)
& -0.0145 & 0.0057 & 0.0070 & (-0.0146,-0.0144) \\

7 & 7
& -0.0076 & 0.0071 & 0.0079 & (-0.0076,-0.0075)
& -0.0085 & 0.0062 & 0.0070 & (-0.0086,-0.0084)
& -0.0091 & 0.0056 & 0.0063 & (-0.0092,-0.0091)
& -0.0096 & 0.0051 & 0.0059 & (-0.0096,-0.0095) \\

\midrule

10 & 8
& -0.0088 & 0.0075 & 0.0084 & (-0.0088,-0.0087)
& -0.0098 & 0.0065 & 0.0074 & (-0.0098,-0.0097)
& -0.0105 & 0.0058 & 0.0067 & (-0.0106,-0.0104)
& -0.0109 & 0.0054 & 0.0063 & (-0.0110,-0.0108) \\

10 & 9
& -0.0059 & 0.0067 & 0.0073 & (-0.0060,-0.0059)
& -0.0069 & 0.0058 & 0.0063 & (-0.0069,-0.0068)
& -0.0074 & 0.0052 & 0.0058 & (-0.0075,-0.0074)
& -0.0078 & 0.0048 & 0.0054 & (-0.0079,-0.0078) \\

10 & 10
& -0.0042 & 0.0060 & 0.0063 & (-0.0042,-0.0042)
& -0.0050 & 0.0052 & 0.0056 & (-0.0050,-0.0050)
& -0.0055 & 0.0047 & 0.0051 & (-0.0055,-0.0055)
& -0.0059 & 0.0044 & 0.0048 & (-0.0059,-0.0058) \\

\midrule

13 & 11
& -0.0047 & 0.0063 & 0.0067 & (-0.0048,-0.0047)
& -0.0056 & 0.0054 & 0.0059 & (-0.0057,-0.0056)
& -0.0061 & 0.0049 & 0.0054 & (-0.0061,-0.0061)
& -0.0064 & 0.0046 & 0.0051 & (-0.0065,-0.0064) \\

13 & 12
& -0.0035 & 0.0057 & 0.0060 & (-0.0035,-0.0034)
& -0.0043 & 0.0049 & 0.0052 & (-0.0043,-0.0042)
& -0.0047 & 0.0045 & 0.0048 & (-0.0047,-0.0047)
& -0.0050 & 0.0042 & 0.0045 & (-0.0050,-0.0050) \\

13 & 13
& -0.0026 & 0.0052 & 0.0054 & (-0.0026,-0.0026)
& -0.0033 & 0.0045 & 0.0047 & (-0.0033,-0.0032)
& -0.0037 & 0.0041 & 0.0043 & (-0.0037,-0.0037)
& -0.0040 & 0.0038 & 0.0041 & (-0.0040,-0.0040) \\

\bottomrule
\end{tabular}
\end{sidewaystable}

\newpage

\begin{sidewaystable}
\centering
\caption{Simulation results for $\widehat J(T_{k|n:G})$ under the Lomax distribution with parameters $\beta=3.5$ and $\eta=2.5$, showing the Mean Estimate, Bias, RMSE, and 95\% CIs for different values of $n$, $k$, and sample size $N$.}
\label{tab2}

\renewcommand{\arraystretch}{1.5}
\setlength{\tabcolsep}{3pt}

\scriptsize

\begin{tabular}{cccccccccccccccccc}
\toprule

& &
\multicolumn{4}{c}{$N=20$} &
\multicolumn{4}{c}{$N=30$} &
\multicolumn{4}{c}{$N=40$} &
\multicolumn{4}{c}{$N=50$} \\

\cmidrule(lr){3-6}
\cmidrule(lr){7-10}
\cmidrule(lr){11-14}
\cmidrule(lr){15-18}

$n$ & $k$
& Mean & Bias & RMSE & 95\% CI
& Mean & Bias & RMSE & 95\% CI
& Mean & Bias & RMSE & 95\% CI
& Mean & Bias & RMSE & 95\% CI \\

\midrule

5 & 3
& -0.0935 & 0.0159 & 0.0352 & (-0.0941,-0.0929)
& -0.0957 & 0.0137 & 0.0298 & (-0.0962,-0.0951)
& -0.0973 & 0.0121 & 0.0265 & (-0.0978,-0.0969)
& -0.0982 & 0.0112 & 0.0241 & (-0.0986,-0.0978) \\

5 & 4
& -0.0418 & 0.0162 & 0.0226 & (-0.0421,-0.0415)
& -0.0440 & 0.0140 & 0.0197 & (-0.0443,-0.0437)
& -0.0455 & 0.0125 & 0.0177 & (-0.0457,-0.0452)
& -0.0467 & 0.0113 & 0.0162 & (-0.0469,-0.0465) \\

5 & 5
& -0.0221 & 0.0147 & 0.0173 & (-0.0223,-0.0219)
& -0.0242 & 0.0126 & 0.0151 & (-0.0243,-0.0240)
& -0.0255 & 0.0113 & 0.0137 & (-0.0256,-0.0253)
& -0.0265 & 0.0103 & 0.0126 & (-0.0266,-0.0263) \\

\midrule

7 & 5
& -0.0362 & 0.0166 & 0.0217 & (-0.0365,-0.0359)
& -0.0387 & 0.0141 & 0.0189 & (-0.0389,-0.0384)
& -0.0405 & 0.0124 & 0.0168 & (-0.0407,-0.0402)
& -0.0413 & 0.0115 & 0.0156 & (-0.0415,-0.0411) \\

7 & 6
& -0.0209 & 0.0147 & 0.0170 & (-0.0211,-0.0207)
& -0.0233 & 0.0124 & 0.0148 & (-0.0234,-0.0231)
& -0.0245 & 0.0112 & 0.0135 & (-0.0246,-0.0243)
& -0.0252 & 0.0104 & 0.0125 & (-0.0254,-0.0251) \\

7 & 7
& -0.0132 & 0.0129 & 0.0141 & (-0.0133,-0.0130)
& -0.0150 & 0.0110 & 0.0124 & (-0.0151,-0.0149)
& -0.0161 & 0.0099 & 0.0112 & (-0.0162,-0.0160)
& -0.0168 & 0.0092 & 0.0105 & (-0.0169,-0.0167) \\

\midrule

10 & 8
& -0.0152 & 0.0136 & 0.0151 & (-0.0153,-0.0151)
& -0.0171 & 0.0117 & 0.0132 & (-0.0172,-0.0170)
& -0.0184 & 0.0104 & 0.0120 & (-0.0185,-0.0182)
& -0.0192 & 0.0096 & 0.0112 & (-0.0193,-0.0191) \\

10 & 9
& -0.0104 & 0.0120 & 0.0129 & (-0.0105,-0.0103)
& -0.0120 & 0.0104 & 0.0113 & (-0.0121,-0.0120)
& -0.0130 & 0.0094 & 0.0104 & (-0.0131,-0.0129)
& -0.0137 & 0.0087 & 0.0097 & (-0.0138,-0.0136) \\

10 & 10
& -0.0073 & 0.0108 & 0.0113 & (-0.0074,-0.0073)
& -0.0088 & 0.0093 & 0.0100 & (-0.0089,-0.0087)
& -0.0096 & 0.0085 & 0.0091 & (-0.0097,-0.0096)
& -0.0102 & 0.0079 & 0.0085 & (-0.0103,-0.0101) \\

\midrule

13 & 11
& -0.0082 & 0.0113 & 0.0120 & (-0.0083,-0.0082)
& -0.0097 & 0.0098 & 0.0106 & (-0.0098,-0.0097)
& -0.0107 & 0.0089 & 0.0096 & (-0.0108,-0.0107)
& -0.0114 & 0.0082 & 0.0090 & (-0.0114,-0.0113) \\

13 & 12
& -0.0061 & 0.0102 & 0.0106 & (-0.0062,-0.0061)
& -0.0074 & 0.0090 & 0.0094 & (-0.0074,-0.0073)
& -0.0083 & 0.0081 & 0.0086 & (-0.0083,-0.0082)
& -0.0089 & 0.0075 & 0.0080 & (-0.0089,-0.0088) \\

13 & 13
& -0.0045 & 0.0094 & 0.0096 & (-0.0046,-0.0045)
& -0.0057 & 0.0082 & 0.0085 & (-0.0058,-0.0057)
& -0.0065 & 0.0074 & 0.0078 & (-0.0066,-0.0065)
& -0.0070 & 0.0069 & 0.0073 & (-0.0070,-0.0069) \\

\bottomrule
\end{tabular}
\end{sidewaystable}

\newpage

\begin{sidewaystable}
\centering
\caption{Simulation results for $\widehat J(T_{k|n:G})$ under the Lomax distribution with parameters $\beta=1.5$ and $\eta=4.5$, showing the Mean Estimate, Bias, RMSE, and 95\% CIs for different values of $n$, $k$, and sample size $N$.}
\label{tab3}

\renewcommand{\arraystretch}{1.5}
\setlength{\tabcolsep}{3pt}

\scriptsize

\begin{tabular}{cccccccccccccccccc}
\toprule

& &
\multicolumn{4}{c}{$N=20$} &
\multicolumn{4}{c}{$N=30$} &
\multicolumn{4}{c}{$N=40$} &
\multicolumn{4}{c}{$N=50$} \\

\cmidrule(lr){3-6}
\cmidrule(lr){7-10}
\cmidrule(lr){11-14}
\cmidrule(lr){15-18}

$n$ & $k$
& Mean & Bias & RMSE & 95\% CI
& Mean & Bias & RMSE & 95\% CI
& Mean & Bias & RMSE & 95\% CI
& Mean & Bias & RMSE & 95\% CI \\

\midrule

5 & 3
& -0.4769 & 0.0304 & 0.1837 & (-0.4805,-0.4734)
& -0.4740 & 0.0333 & 0.1486 & (-0.4769,-0.4712)
& -0.4726 & 0.0348 & 0.1304 & (-0.4750,-0.4701)
& -0.4728 & 0.0346 & 0.1157 & (-0.4749,-0.4706) \\

5 & 4
& -0.1996 & 0.0593 & 0.1008 & (-0.2012,-0.1980)
& -0.2074 & 0.0515 & 0.0872 & (-0.2087,-0.2060)
& -0.2117 & 0.0472 & 0.0782 & (-0.2129,-0.2105)
& -0.2162 & 0.0427 & 0.0709 & (-0.2173,-0.2151) \\

5 & 5
& -0.1036 & 0.0571 & 0.0726 & (-0.1044,-0.1027)
& -0.1122 & 0.0485 & 0.0635 & (-0.1130,-0.1114)
& -0.1162 & 0.0445 & 0.0580 & (-0.1169,-0.1154)
& -0.1199 & 0.0408 & 0.0532 & (-0.1206,-0.1192) \\

\midrule

7 & 5
& -0.1720 & 0.0620 & 0.0960 & (-0.1734,-0.1705)
& -0.1804 & 0.0535 & 0.0824 & (-0.1816,-0.1792)
& -0.1850 & 0.0489 & 0.0740 & (-0.1861,-0.1839)
& -0.1885 & 0.0454 & 0.0688 & (-0.1895,-0.1875) \\

7 & 6
& -0.0977 & 0.0577 & 0.0723 & (-0.0986,-0.0968)
& -0.1063 & 0.0491 & 0.0634 & (-0.1071,-0.1055)
& -0.1107 & 0.0448 & 0.0572 & (-0.1114,-0.1100)
& -0.1141 & 0.0413 & 0.0529 & (-0.1147,-0.1134) \\

7 & 7
& -0.0606 & 0.0519 & 0.0591 & (-0.0612,-0.0601)
& -0.0680 & 0.0445 & 0.0519 & (-0.0685,-0.0675)
& -0.0727 & 0.0398 & 0.0469 & (-0.0732,-0.0723)
& -0.0749 & 0.0376 & 0.0442 & (-0.0754,-0.0745) \\

\midrule

10 & 8
& -0.0701 & 0.0544 & 0.0634 & (-0.0708,-0.0695)
& -0.0776 & 0.0469 & 0.0558 & (-0.0782,-0.0770)
& -0.0825 & 0.0421 & 0.0505 & (-0.0830,-0.0819)
& -0.0860 & 0.0385 & 0.0470 & (-0.0866,-0.0855) \\

10 & 9
& -0.0475 & 0.0489 & 0.0540 & (-0.0480,-0.0471)
& -0.0543 & 0.0421 & 0.0475 & (-0.0548,-0.0539)
& -0.0583 & 0.0382 & 0.0434 & (-0.0587,-0.0579)
& -0.0608 & 0.0357 & 0.0406 & (-0.0612,-0.0604) \\

10 & 10
& -0.0332 & 0.0444 & 0.0474 & (-0.0335,-0.0329)
& -0.0393 & 0.0383 & 0.0417 & (-0.0396,-0.0390)
& -0.0428 & 0.0348 & 0.0382 & (-0.0431,-0.0425)
& -0.0455 & 0.0321 & 0.0354 & (-0.0458,-0.0452) \\

\midrule

13 & 11
& -0.0379 & 0.0461 & 0.0497 & (-0.0382,-0.0375)
& -0.0442 & 0.0397 & 0.0438 & (-0.0446,-0.0439)
& -0.0478 & 0.0362 & 0.0403 & (-0.0482,-0.0475)
& -0.0503 & 0.0337 & 0.0376 & (-0.0506,-0.0499) \\

13 & 12
& -0.0277 & 0.0420 & 0.0443 & (-0.0280,-0.0275)
& -0.0332 & 0.0366 & 0.0392 & (-0.0335,-0.0329)
& -0.0368 & 0.0329 & 0.0358 & (-0.0371,-0.0366)
& -0.0392 & 0.0306 & 0.0334 & (-0.0394,-0.0389) \\

13 & 13
& -0.0206 & 0.0386 & 0.0400 & (-0.0208,-0.0204)
& -0.0256 & 0.0336 & 0.0353 & (-0.0259,-0.0254)
& -0.0289 & 0.0303 & 0.0323 & (-0.0291,-0.0287)
& -0.0310 & 0.0282 & 0.0302 & (-0.0312,-0.0308) \\

\bottomrule
\end{tabular}
\end{sidewaystable}
\subsection{Real data illustration}
Let $X_1,X_2,\ldots,X_{46}$ denote the observed COVID--19 vaccination rates (measured as doses administered per 100 individuals) collected from 46 countries in southern Africa. These observations were previously analyzed in \cite{almongy2022}. The recorded data are given below:
0.042, 0.205, 0.285, 0.319, 0.464, 0.550, 0.889, 0.895, 0.939, 0.986, 1.000, 1.088, 1.212, 1.244, 1.450, 1.593, 1.844, 2.039, 2.157, 2.167, 2.334, 2.440, 2.657, 3.685, 3.879, 4.493, 4.800, 4.944, 5.155, 5.674, 7.602, 10.004, 12.238, 12.520, 12.553, 13.063, 15.105, 15.229, 15.629, 15.848, 18.641, 18.940, 29.885, 58.162, 61.838, 72.286.

To model the data, we examine four continuous probability distributions, namely Weibull, Normal, Gamma, and Lomax distributions, with corresponding PDFs given as follows:
\begin{itemize}
\item \textbf{Weibull distribution}
$$
f(x;\alpha_1,\alpha_2)=\frac{\alpha_1}{\alpha_2}
\left(\frac{x}{\alpha_2}\right)^{\alpha_1-1}
\exp\left[-\left(\frac{x}{\alpha_2}\right)^{\alpha_1}\right],
\qquad x\geq 0.
$$
\item \textbf{Normal distribution}
$$
f(x;\mu,\sigma)=\frac{1}{\sigma\sqrt{2\pi}}
\exp\left[-\frac{(x-\mu)^2}{2\sigma^2}\right],
\qquad x\in\mathbb{R}.
$$
\item \textbf{Gamma distribution}
$$
f(x;k,r)=\frac{r^k}{\Gamma(k)}x^{k-1}e^{-rx},
\qquad x>0.
$$
\item \textbf{Lomax distribution}
$$
f(x;\beta,\eta)=\frac{\beta}{\eta}
\left(1+\frac{x}{\eta}\right)^{-(\beta+1)},
\qquad x\geq 0.
$$
\end{itemize}

The suitability of these models is assessed through the Kolmogorov--Smirnov goodness-of-fit procedure. The estimated parameters together with the corresponding AIC, BIC, and K--S $p$-values are presented in Table~\ref{tab:fit}.
\begin{table}[H] \centering \caption{Goodness-of-Fit results.} \label{tab:fit} \begin{tabular}{lcccc} \toprule \textbf{Distribution} & \textbf{Parameters} & \textbf{AIC} & \textbf{BIC} & \textbf{$p$-value} \\ \midrule Weibull & $\alpha_1=0.6857,\alpha_2=7.3422$ & 291.6627 & 295.3200 & 0.6215 \\ Normal & $\mu=9.8037,\sigma=15.8544$ & 388.7797 & 392.4370 & 0.0020 \\ Gamma & $k=0.5818,r=0.0593$ & 294.7664 & 298.4237 & 0.3131 \\ Lomax & $\beta=1.3938,\eta=5.585$ & 289.7433 & 293.4006 & 0.6415 \\ \bottomrule \end{tabular} \end{table}
Among the fitted models, the Lomax distribution provides the most adequate representation of the data, as it attains the minimum AIC and BIC values together with the largest K--S $p$-value. Hence, the Lomax model is selected for further analysis.
Using the fitted Lomax parameters, we evaluate the theoretical CREx associated with the $\mathcal C(k|n{:}G)$ systems. In addition, the non-parametric estimator introduced in \eqref{eq3.1} is computed from the observed sample.

The numerical findings reported in Table~\ref{tab4} indicate that the estimated CREx values are generally in close agreement with their corresponding theoretical values. The approximation improves noticeably as the value of $k$ increases, which supports the behavior observed in the simulation study. However, a moderate discrepancy is observed for the cases $(n,k)=(5,3)$, $(5,4)$, and $(5,5)$. This variation may arise from the greater influence of tail observations on systems with smaller configurations, where non-parametric estimation becomes comparatively less stable. Overall, the estimator demonstrates satisfactory performance and effectively captures the theoretical behavior of the proposed measure. 
\begin{sidewaystable}[htbp]
\centering
\caption{Comparison of estimated  and theoretical values of CREx with absolute error for COVID-19 vaccination rates data.}
\label{tab4}

\renewcommand{\arraystretch}{1.5}
\setlength{\tabcolsep}{7pt}

\scriptsize

\begin{tabular}{c ccc ccc ccc ccc}
\toprule

\multirow{2}{*}{$k$}

& \multicolumn{3}{c}{$n=5$}
& \multicolumn{3}{c}{$n=7$}
& \multicolumn{3}{c}{$n=10$}
& \multicolumn{3}{c}{$n=13$} \\

\cmidrule(lr){2-4}
\cmidrule(lr){5-7}
\cmidrule(lr){8-10}
\cmidrule(lr){11-13}

& $\hat J(T_{k|n:G})$
& $J(T_{k|n:G})$
& Abs.\ Error

& $\hat J(T_{k|n:G})$
& $J(T_{k|n:G})$
& Abs.\ Error

& $\hat J(T_{k|n:G})$
& $J(T_{k|n:G})$
& Abs.\ Error

& $\hat J(T_{k|n:G})$
& $J(T_{k|n:G})$
& Abs.\ Error
\\

\midrule

3
& -0.5867 & -0.6871 & 0.1004
& -1.0125 & -1.1275 & 0.1150
& -1.9078 & -2.0368 & 0.1290
& -3.1109 & -3.2446 & 0.1337
\\

4
& -0.2669 & -0.3487 & 0.0818
& -0.4303 & -0.5388 & 0.1085
& -0.7729 & -0.9313 & 0.1584
& -1.2323 & -1.4527 & 0.2204
\\

5
&         &         &
& -0.2358 & -0.3147 & 0.0789
& -0.4038 & -0.5188 & 0.1150
& -0.6284 & -0.7898 & 0.1614
\\

6
&         &         &
& -0.1493 & -0.2087 & 0.0594
& -0.2451 & -0.3281 & 0.0830
& -0.3729 & -0.4867 & 0.1138
\\

7
&         &         &
&         &         &
& -0.1633 & -0.2266 & 0.0633
& -0.2434 & -0.3273 & 0.0839
\\

8
&         &         &
&         &         &
& -0.1159 & -0.1670 & 0.0511
& -0.1695 & -0.2349 & 0.0654
\\

9
&         &         &
&         &         &
& -0.0861 & -0.1293 & 0.0432
& -0.1238 & -0.1772 & 0.0534
\\

10
&         &         &
&         &         &
&         &         &
& -0.0937 & -0.1390 & 0.0453
\\

11
&         &         &
&         &         &
&         &         &
& -0.0729 & -0.1125 & 0.0396
\\

12
&         &         &
&         &         &
&         &         &
& -0.0581 & -0.0934 & 0.0353
\\

\bottomrule
\end{tabular}
\end{sidewaystable}

\section{Conclusion}
In this study, we investigated CREx of $\mathcal C(k|n{:}G)$ systems. An explicit expression for  CREx measure was derived for systems consisting of iid components and illustrated through a suitable example. The obtained expressions provide a useful framework for understanding the uncertainty and information characteristics associated with consecutive $\mathcal C(k|n{:}G)$ systems.
Further, several important properties of CREx were explored under different stochastic ordering relations.  We also derived several upper and lower bounds for the CREx of consecutive systems, and numerical illustrations were provided. A significant contribution of the paper lies in the characterization results developed through the CREx measure. In particular, characterization results associated with the dispersive order were established using CREx of $\mathcal C(k|n{:}G)$ systems. It was shown that, for all systems satisfying the condition $(2k \geq n),$ equality of the proposed measure holds if and only if the underlying parent distributions differ by a location shift. In addition, characterization results involving both location and scale transformations were also obtained. Furthermore, we extended CREx to the dynamic version for $\mathcal C(k|n{:}G)$ systems and established a relationship between DCREx and MRL.
To demonstrate the practical applicability of the proposed methodology, a non-parametric estimator of CREx was proposed and its consistency property was established. The finite-sample performance of the estimator was further examined through extensive Monte Carlo simulation studies under different system configurations. Moreover, the usefulness of the proposed methodology was illustrated through real data analysis, confirming its effectiveness in practical reliability applications. 

The results presented in this work open several directions for future research, including the study of a weighted version of CREx for $\mathcal C(k|n{:}G)$ systems, a CDF-based extension, and broader applications in statistical learning and information-theoretic image analysis.
\subsection*{Author's contributions}
Aman Pandey:  Conceptualization, data analysis, interpretation of the results, and drafting of the initial manuscript.
\noindent\\
Chanchal Kundu: Supervision, study conception and design, critical revision of the manuscript, and final approval of the version to be published.
\subsection*{Funding}
No funding was received for conducting this study.
\subsection*{Data availability}
All the datasets used in this study have been properly cited within the manuscript.
\subsection*{Statements and declarations}
\textbf{Competing interests:}
	On behalf of all authors, the corresponding author states that there is no conflict of interest.


\begin{thebibliography}{99}

\bibitem[Almongy et al.(2022)]{almongy2022}
Almongy, H. M., Almetwally, E. M., Haj Ahmad, H., and Al-nefaie, A. (2022). Modeling of COVID-19 vaccination rate using odd Lomax inverted Nadarajah--Haghighi distribution. \textit{PLoS ONE}, \textbf{17}(10), 0276181.

\bibitem[Chakraborty and Pradhan(2024)]{chak2024}
Chakraborty, S., and Pradhan, B. (2024). On cumulative residual extropy of coherent and mixed systems. \textit{Annals of Operations Research}, \textbf{340}, 59--81.

\bibitem[Chakraborty and Pradhan(2025)]{chak2025}
Chakraborty, S., and Pradhan, B. (2025). On estimation of cumulative residual extropy and its quantile version. \textit{Ricerche di Matematica}, \textbf{74}. 1165--1176.

\bibitem[Dembińska and Jasiński(2025)]{dem2025}Dembińska, A. and Jasiński, K. (2025). Maximum likelihood inference about parameters of geometric lifetimes of heterogeneous components from data collected till failure of a $k$-out-of-$n{:}G$ system. \textit{Journal of Computational and Applied Mathematics}, \textbf{454}, 116195.


\bibitem[Eryilmaz(2009)]{ery2009}
Eryilmaz, S. (2009). Reliability properties of consecutive $k$-out-of-$n$ systems of arbitrarily dependent components. \textit{Reliability Engineering \& System Safety}, \textbf{94}(2), 350--356.

\bibitem[Eryilmaz(2010)]{ery2010}Eryilmaz, S. (2010). Conditional lifetimes of consecutive  $k$-out-of-$n$  Systems, \textit{IEEE Transactions on Reliability}. \textbf{59}(1), 178--182.

\bibitem[Eryilmaz(2026)]{ery2026}
Eryilmaz, S. (2026). On reliability of consecutive $k$-out-of-$n:G$ system equipped with protection blocks. \textit{International Journal of General Systems}, \textbf{55}(2), 129--144.

\bibitem[Eryilmaz and Navarro(2012)]{ery2012}
Eryilmaz, S., and Navarro, J. (2012). Failure rates of consecutive $k$-out-of-$n$ systems. \textit{Journal of the Korean Statistical Society}, \textbf{41}(1), 1--11.

\bibitem[Jahanshahi et al.(2019)]{jahan2019}
Jahanshahi, S. M. A., Zarei, H., and Khammar, A. H. (2019). On cumulative residual extropy. \textit{Probability in the Engineering and Informational Sciences}, \textbf{34}(4), 605--625.

\bibitem[Jewitt(1989)]{jew1989}
Jewitt, I. (1989). Choosing between risky prospects: The characterization of comparative statics results, and location independent risk. \textit{Management Science}, \textbf{35}(1), 60--70.

\bibitem[Kayid and Alshehri(2024)]{kayid2024}
Kayid, M., and Alshehri, M. A. (2024). Shannon differential entropy properties of consecutive $k$-out-of-$n:G$ systems. \textit{Operations Research Letters}, \textbf{57}, 107190.

\bibitem[Kayid and Balakrishnan(2025)]{kayid2025}
Kayid, M., and Balakrishnan, N. (2025). Cumulative residual entropy of linear consecutive $k$-out-of-$n:G$ systems and their applications. \textit{Methodology and Computing in Applied Probability}, \textbf{27}(2), 48.

\bibitem[Kazemi et al.(2021)]{kaz2021}Kazemi, M.R., Tahmasebi, S., Calì, C. and Longobardi, M. (2021). Cumulative residual extropy of minimum ranked set sampling with unequal samples, \textit{Results in Applied Mathematics}, 10, 100156.

\bibitem[Kattumannil and Sreedevi(2022)]{kattu2022}Kattumannil, S.K. and Sreedevi, E.P. (2022). Non-parametric estimation of cumulative (residual) extropy, \textit{Statistics \& Probability Letters}, 185, 109434.

\bibitem[Lad et al.(2015)]{lad2015}
Lad, F., Sanfilippo, G., and Agro, G. (2015). Extropy: Complementary dual of entropy. \textit{Statistical Science}, \textbf{30}(1), 40--58.

\bibitem[Landsberger and Meilijson(1994)]{land1994}
Landsberger, M., and Meilijson, I. (1994). The generating process and an extension of Jewitt’s location independent risk concept. \textit{Management Science}, \textbf{40}(5), 662--669.

\bibitem[Noughabi and Arghami(2011)]{nou2011}
Noughabi, H. A., and Arghami, N. R. (2011). Testing exponentiality based on characterizations of the exponential distribution. \textit{Journal of Statistical Computation and Simulation}, \textbf{81}(11), 1641--1651.

\bibitem[Pakdaman and Noughabi(2025)]{pakda2025}Pakdaman Z., Noughabi R.A. (2025). On the study of the cumulative residual extropy of mixed used systems and their complexity. \textit{Probability in the Engineering and Informational Sciences}, \textbf{39}(1):122-140. 

\bibitem[Qiu et al.(2019)]{qiu2019}
Qiu, G., Wang, L., and Wang, X. (2019). On extropy properties of mixed systems. \textit{Probability in the Engineering and Informational Sciences}, \textbf{33}(3), 471--486.

\bibitem[Sathar and Nair(2021)]{sathar2021}
Sathar, E. I. A., and Nair, R. D. (2021). On dynamic survival extropy. \textit{Communications in Statistics -- Theory and Methods}, \textbf{50}(6), 1295--1313.

\bibitem[Shaked and Shanthikumar(2007)]{shaked2007}
Shaked, M., and Shanthikumar, J. G. (2007). \textit{Stochastic Orders}. Springer, New York.

\bibitem[Shannon(1948)]{shannon1948}
Shannon, C. E. (1948). A mathematical theory of communication. \textit{Bell System Technical Journal}, \textbf{27}(3), 379--423.

\bibitem[Tian and Lu(2026)]{tian2026}
Tian, X., and Lu, B. (2026). Weighted Shannon differential entropy for consecutive $k$-out-of-$n:G$ systems: Properties, bounds, estimation and applications. \textit{Methodology and Computing in Applied Probability}, \textbf{28}, 45.

\bibitem[Vasicek(1976)]{vas1976}
Vasicek, O. (1976). A test for normality based on sample entropy. \textit{Journal of the Royal Statistical Society: Series B}, \textbf{38}(1), 54--59.

\bibitem[Yi et al.(2026)]{yi2026}
Yi, H., Balakrishnan, N., and Li, X. (2026). A general type of linear consecutive-$k$ systems. \textit{Methodology and Computing in Applied Probability}, \textbf{28}, 14.
\end{thebibliography}
\end{document}